\documentclass{tran-l}

\usepackage{amssymb}
\usepackage{amsmath}

\usepackage{diagrams}

\newtheorem{theorem}{Theorem}[section]

\newtheorem{corollary}[theorem]{Corollary}

\newtheorem{definition}[theorem]{Definition}
\newtheorem{example}[theorem]{Example}

\newtheorem{lemma}[theorem]{Lemma}

\newtheorem{proposition}[theorem]{Proposition}

\newarrow {Partial} ----{harpoon}
\newarrow{Into} C--->
\newarrow{Onto} ----{>>}

\begin{document}

\title{Stability of properties of locales under groups}
\author{Christopher Townsend}
\maketitle
\begin{abstract}
Given a particular collection of categorical axioms, aimed at capturing properties of the category of locales, we show that if $\mathcal{C}$ is a category that satisfies the axioms then so too is the category $[ G, \mathcal{C}]$ of $G$-objects, for any internal group $G$. To achieve this we prove a general categorical result: if an object $S$ is double exponentiable in a category with finite products then so is its associated trivial $G$-object $(S, \pi_2: G \times S \rTo S)$. The result holds even if $S$ is not exponentiable.

An example is given of a category $\mathcal{C}$ that satisfies the axioms, but for which there is no elementary topos $\mathcal{E}$ such that $\mathcal{C}$ is the category of locales over $\mathcal{E}$. 

It is shown, in outline, how the results can be extended from groups to groupoids.

\end{abstract}

\section{Introduction}

%
%
%

Given a category $\mathcal{C}$ with finite products and an internal group $G$, a categorical axiom is said to be \emph{$G$-stable} provided if it is true of $\mathcal{C}$ then so too is it true of $[ G , \mathcal{C}]$, the category of $G$-objects. An example is the property of having equalizers. A non-example is the property `every epimorphism splits' which holds in the category $\mathbf{Set}$, if the axiom of choice is true, but for any group $G$, $G \rTo 1$ is an epimorphic $G$-homomorphism which is split if and only if $G$ is trivial.

A set of categorical axioms, investigated in \cite{towhofman}, \cite{closedsubgroup}, \cite{towaxioms} and \cite{towslice}, captures various properties of the category of locales. Certain aspects of locale theory can be developed axiomatically: proper and open maps are pullback stable, the Hofmann-Mislove result shown, the closed subgroup theorem holds, Plewe's result that triqutient surjections are of effective descent proved, the patch construction developed, etc. The purpose of this paper is to explore the question of whether the axioms are $G$-stable for an internal group $G$. The answer is that, with a minor modification that does not weaken the theory, the axioms are $G$-stable. The minor modification is that the existence of coequalizers is no longer an axiom. Intuitively a modification of this sort is needed as constructing coequalizers in $[G,\mathcal{C}]$ appears to require coequalizers that are stable under products, and this stability is an additional property not true of the category of locales. 

Once we have established that the axioms are $G$-stable we then establish a new result which is that not every category that satisfies the axioms is a category of locales for some topos. Any open or compact localic group that is not \'{e}tale complete (in the sense of Moerdijk, e.g. Section 7 of \cite{MoerClassTop}) provides an example.

Our next step is to verify that even without any coequalizers in $\mathcal{C}$, key results about coequalizers still hold. Specifically we show that triquotient surjections are coequalizers and that for every open or compact group $G$, there is a connected components adjunction $[ G , \mathcal{C}] \pile{\rTo \\ \lTo} \mathcal{C}$. 

Finally we include some comments on how it is easy to extend the results from internal groups to groupoids, given that the axioms are slice stable (\cite{towslice}).

\section{Preliminary categorical definitions and main categorical result}\label{prelim}
Let $\mathcal{C}$ be a category with finite products and $S$ an object of $\mathcal{C}$. We use the notation $S^X$ for the presheaf 
\begin{eqnarray*}
\mathcal{C}^{op} &\rTo &\mathbf{Set} \\
Y &\rMapsto &\mathcal{C}(Y\times X,S)
\end{eqnarray*}
It can be verified, using Yoneda's lemma, that $S^{X}$ is the exponential ${\bf y}S^{ {\bf y}X}$ in the presheaf category $[\mathcal{C}^{op},\mathbf{Set}]$, so the notation is reasonable (where $\bf{y}$ is the Yoneda embedding). We use $\mathcal{C}_S^{op}$ as notation for the full subcategory of $[\mathcal{C}^{op},\mathbf{Set}]$ consisting of objects of the form $S^X$; there is a contravariant functor $\mathcal{C} \rTo^{S^{(\_)}} \mathcal{C}_S^{op}$

Our first lemma is rather simple.
\begin{lemma}\label{hannah}
Let $\mathcal{C}$ be a category with finite products and $\mathbb{T}=(T,\eta,\mu)$ a monad on $\mathcal{C}$. Then for any two $\mathbb{T}$-algebras $(X,a: TX \rTo X)$ and $(S,s: TS \rTo S)$ the diagram 
\begin{eqnarray*}
(S,s)^{(X,a)}\rTo^{(S,s)^a} (S,s)^{(TX,\mu_X)} \pile{\rTo^{(S,s)^{\mu_X}} \\ \rTo_{(S,s)^{Ta}}} (S,s)^{(TTX,\mu_{TX})}
\end{eqnarray*}
is an equalizer in $(\mathcal{C}^{\mathbb{T}})^{op}_{(S,s)}$.
\end{lemma}
\begin{proof}
If $\epsilon: (S,s)^{(Y,b)} \rTo (S,s)^{(TX,\mu_X)}$ is a natural transformation such that $(S,s)^{\mu_X}\epsilon=(S,s)^{Ta}\epsilon$ then define $\bar{\epsilon}:(S,s)^{(Y,b)}\rTo (S,s)^{(X,a)}$ by setting $\bar{\epsilon}_{(Z,c)}(u)$ to
\begin{eqnarray*}
Z \times X \rTo^{Id_Z \times \eta_X} Z \times TX \rTo^{\epsilon_{(Z,c)}(u)}S\text{.}
\end{eqnarray*}
This is well defined (i.e. defines a $\mathbb{T}$-algebra homomorphism from $(Z,c) \times (X,a)$ to $(S,s)$) because 
\begin{eqnarray*}
(Z,c) \times (TTX,\mu_X) \pile{ \rTo^{Id_Z \times \mu_X} \\ \rTo_{Id_Z \times Ta}} (Z,c) \times (TX, \mu_X) \rTo^{Id_Z \times a} (Z,c) \times (X,a)
\end{eqnarray*}
is a coequalizer in $\mathcal{C}^{\mathbb{T}}$ (it is $U$-split, by $Id_Z \times \eta_X$ and $Id_Z \times \eta_{Ta}$, where $U: \mathcal{C}^{\mathbb{T}} \rTo \mathcal{C}$ is the forgetful functor and $U$, being monadic, creates coequalizers for $U$-split forks).
\end{proof}
Recall that an adjunction $L\dashv R:\mathcal{D}\pile{\rTo \\ \lTo} \mathcal{C}$ between categories, both with finite products, satisfies \emph{Frobenius reciprocity} provided the map
$L(R(X)\times W)\rTo^{(L\pi_{1},L\pi_{2})}LRX\times LW\rTo^{\varepsilon _{X}\times Id_{LW}}X\times LW$ is an isomorphism for all objects $W$ and $X$ of $\mathcal{D}$
and $\mathcal{C}$ respectively. For example any morphism $f:X\rTo Y$ of a cartesian category $\mathcal{C}$ gives rise to a pullback adjunction $\Sigma_f \dashv f^*: \mathcal{C}/X \rTo \mathcal{C}/X$ that satisfies Frobenius reciprocity. For another example if $G=(G,m:G \times G \rTo G, e : 1 \rTo G, i:G \rTo G)$ is a group object in a category $\mathcal{C}$ with finite products, then the adjunction $G  \times (\_) \dashv U : \mathcal{C} \rTo  $$ [ G , \mathcal{C}]$ satisfies Frobenius reciprocity. Here $G \times (\_)$ sends an object $X$  of $\mathcal{C}$ to the $G$-object $(G\times X, m \times Id_X)$ and $U$ is the forgetful functor (forget the $G$ action). The counit of this adjunction, at a $G$-object $(X,a)$, is given by the $G$-homomorphism $a : G \times X \rTo X$ so establishing Frobenius reciprocity for the adjunction amounts to finding, for any object $Y$ of $\mathcal{C}$ and any $G$-object $(X,a)$, an inverse for $(G \times X \times Y, m \times Id_X \times Id_Y) \rTo^{(a\pi_{23},\pi_{13})} (X,a) \times (G \times Y,m \times Id_Y)$. The inverse is given by $  X \times G \times Y \rTo^{(\pi_2,a(i\pi_2,\pi_1),\pi_3)} G \times X \times Y$. Another way to establish Frobenius reciprocity is to recall that for any $G$-object $(X,a:G\times X \rTo X)$ there is a $G$-isomorphism $(G,m) \times (X, \pi_2) \cong (G,m) \times (X,a)$.

The adjunctions $G \times (\_) \dashv  U$ are key to the considerations of this paper so we recall a couple of basic facts: (1) $U$ is monadic and (2) if $G_1$ and $G_2$ are two internal groups then to prove that $G_1$ is isomorphic to $G_2$ it is sufficient to exhibit an equivalence of categories $\psi: [ G_1 , \mathcal{C}] \rTo^{\simeq} $$[ G_2 , \mathcal{C}]$ which commutes with the adjunction; that is, there exists a natural isomorphism $\beta: \psi G_1 \times (\_) \rTo^{\cong} G_2 \times (\_)$. To see (2) notice that $G_i = [U \circ G_i \times (\_)](1)$ for $i=1,2$ and  $\psi$ (together with the natural isomorphism $\beta$) commute with the two monad structures.
The next lemma is a generalisation of change of base:
\begin{lemma}\label{Frob}
Let $\mathcal{C}$ and $\mathcal{D}$ be two categories with finite products and $L: \mathcal{D} \rTo \mathcal{C}$ and $R: \mathcal{C} \rTo \mathcal{D}$ two functors such that $L \dashv R$ and the adjunction satisfies Frobenius reciprocity. Then for any object $S$ of $\mathcal{C}$, $L \dashv R$ extends contravariantly to an adjunction $\mathcal{D}_{RS}^{op}\pile{\rTo \\ \lTo} \mathcal{C}^{op}_S$\end{lemma}
This lemma is essentially originally shown in \cite{towgeom}. In the case that the adjunction is a pullback adjunction arising from a locale map and $S$ is the Sierpi\'{n}ski locale, the morphisms of $\mathcal{C}_{S}^{op}$ can be used to represent dcpo homomorphisms and the adjunction established by the lemma shows how to move dcpo homomorphisms between sheaf toposes; this is how the lemma can be viewed as a generalisation of change of base. Consult \cite{towaxioms} for more detail.
\begin{proof}
Precomposition with $L$ and $R$ defines for any adjunction $L \dashv R$ an adjunction between presheaf categories, $[\mathcal{D}^{op},\mathbf{Set}] \pile{ \rTo \\ \lTo } $$ [\mathcal{C}^{op},\mathbf{Set}]$. But the Frobenius condition implies for $W$ and $X$ of $\mathcal{D}$
and $\mathcal{C}$ respectively that $S^XL\cong RS^{RX}$ and $S^WR\cong {S}^{LW}$ and so the adjunction restricts to $\mathcal{D}_{RS}^{op}\pile{\rTo \\ \lTo} \mathcal{C}^{op}_S$ which can be seen to extend (via $S^{(\_)}$) the adjunction $L \dashv R$. The unit of the extension is given by $S^{\epsilon}$ and the counit by $RS^{\eta}$ where $\eta$ (respectively $\epsilon$) is the unit (counit) of $L \dashv R$.
\end{proof}
It is an exercise, based on the result just given, to verify that if $\delta : RS^{RX} \rTo RS^W$ then the adjoint transpose of $\delta$, written $\bar{\delta}: S^X \rTo S^{LW}$, is defined by setting, for any $u : Z \times X \rTo S$, $\bar{\delta}_Z(u)$ to be
\begin{eqnarray*}
Z \times LW \rTo^{[(\epsilon_Z \times Id_{LW})L(\pi_1,\pi_2)]^{-1}} L(RZ \times W) \rTo^{\widetilde{\delta_Z(Ru)}}S
\end{eqnarray*}
where $\tilde{(\_)}$ is the action of taking adjoint transpose under $L \dashv R$. Given this observation and our observation that the adjunction $G \times (\_) \dashv U : \mathcal{C} \pile{ \rTo \\ \lTo } $$[ G , \mathcal{C} ]$ satisfies Frobenius reciprocity the following corollary is almost immediate:
\begin{corollary}\label{nathan}
Let $\mathcal{C}$ be a category with finite products and $G$ an internal group. For any objects $Y$ and $S$ of $\mathcal{C}$ and $(X,a)$ a $G$-object,
\begin{eqnarray*}
Nat[S^{X},S^Y] \cong Nat[(S,\pi_2)^{(X,a)},(S,\pi_2)^{(G \times Y, m \times Id_Y)}]
\end{eqnarray*}
naturally in both arguments. The mate of $\delta:S^X \rTo S^Y$, evaluated at $u: (Z,c) \times (X,a) \rTo (S,\pi_2)$ is given by
\begin{eqnarray*}
Z \times G \times Y \rTo^{(c(i \pi_2 , \pi_1),\pi_3)} Z \times Y \rTo^{\delta_Z(u)} S
\end{eqnarray*}
i.e. $(z,g,y)$ is in $\bar{\delta}_{(Z,c)}(u)$ if and only if $(g^{-1}z,y)$ is in $\delta(u)$.
\end{corollary}
\begin{proof}
In addition to the comments in the preamble, observe that the adjoint transpose of $\delta_Z(u)$ (under $G \times (\_) \dashv U$) is given by $G \times Z \times Y \rTo^{\pi_{23}} Z \times Y \rTo^{\delta(u)} S$ because $S$ has the trivial action.
\end{proof}

Why are we so interested in these natural transformstions? Essentially because they are by construction the points of the double exponential $S^{S^X}$; if that double exponential exists. The category of locales provides an example of a category where exponentials do not always exists (not all locales are locally compact) but for which double exponentiation (at the Sierpi\'{n}ski locale at least) does always exist, \cite{victow}. Therefore there is good reason to investigate double exponentiation categorically in the absence of an assumption of cartesian closedness and these natural transformations play a central role. Let us make this more precise, beginning with a definition.

\begin{definition}
An object $S$ in a category $\mathcal{C}$ with finite products is \emph{double exponentiable} provided for every other object $X$ the exponential $({\bf y}S)^{S^X}$ exists in $[\mathcal{C}^{op},\mathbf{Set}]$ and is representable.
\end{definition}

If an object is double exponentiable then a strong \emph{double exponential} monad can be defined on $\mathcal{C}$; its functor part sends an object $X$ to the object that represents $({\bf y}S)^{S^X}$ and the rest of the monad structure and the strength are determined by the universal property of the double exponential. The key universal property can be expressed by saying that if $P(X)$ is the functor part of the double exponential monad, evaluated at $X$, then for any other object $Y$, there is a bijection, natural in $X$ and $Y$, between morphisms $Y \rTo P(X)$ and natural transformations $S^X \rTo S^Y$. Notice that if $S$ is double exponentiable, the opposite of the Kleisli category of the double power monad, $\mathcal{C}^{op}_P$, can be identified with $\mathcal{C}^{op}_S$ (i.e. the full subcategory of $[\mathcal{C}^{op},\mathbf{Set}]$ consisting of objects of the form $S^X$). Composition of Kleisli arrows is just composition of natural transformations. We will treat the opposite of the Kleisli category as this full subcategory below without notating the equivalence.

We can now prove a categorical proposition which is of general interest and is the main technical insight of this paper.

\begin{proposition}\label{main}
Let $\mathcal{C}$ be a category with finite products, $G$ an internal group and $S$ a double exponentiable object. Then $(S,\pi_2)$ is a double exponentiable object in $[G,\mathcal{C}]$.
\end{proposition}
\begin{proof}
Let $(X,a)$ be a $G$-object. Our first observation is that $PX$ (i.e. the object representing ${\bf y}S^{S^X}$) can be made into a $G$-object by defining $a^{P}$ to be $G \times PX\rTo^{t_{G,X}} P(G \times X) \rTo^{P(a)}PX$, where $t$ is the strength on $P$. This follows by application of the definition of strength ($t_{1,X}\cong Id_{PX}$ and $t_{X \times Y, Z} = t_{X, Y\times Z} (Id_X \times t_{Y,Z})$). For any other $G$-object $(Y,b)$, $G$-homomorphisms from $Y$ to $PX$ correspond to natural transformations $\delta: S^X \rTo S^Y$ with the property that $\delta^G S^a = S^b\delta$. So to conclude the proof all we need to do is to show that such natural transformations are in bijection with natural transformations $(S,\pi_2)^{(X,a)}\rTo (S,\pi_2)^{(Y,b)}$. (The bijection must be natural in $(Y,b)$, but this aspect is straightforward and is not commented on further.)

Lemma \ref{hannah}, with $\mathbb{T}$ the monad induced by $ G\times (\_) \dashv U$, shows that natural transformations $\epsilon:(S,\pi_2)^{(X,a)}\rTo (S,\pi_2)^{(Y,b)}$ are in (natural) bijection with natural transformations $\epsilon':(S,\pi_2)^{(X,a)}\rTo (S,\pi_2)^{(G \times Y, m \times Id_Y)}$ such that $(S,\pi_2)^{m \times Id_Y} \epsilon' = (S,\pi_2)^{Id_G \times b} \epsilon' $ and the corollary shows 
\begin{eqnarray*}
Nat[S^{X},S^Y] \cong Nat[(S,\pi_2)^{(X,a)},(S,\pi_2)^{(G \times Y, m \times Id_Y)}]\text{.}
\end{eqnarray*}
Since this last bijection is natural we can see that the mate of $S^b\delta$ is $(S,\pi_2)^{Id_G \times b} \bar{\delta}$ where $\bar{\delta}$ is the mate of $\delta:S^X \rTo S^Y$. So to complete the proof all that is required is a verification that $\overline{\delta^G S^a} = (S,\pi_2)^{m \times Id_Y}   \bar{\delta}$.

Say we are given a $G$-homomorphism $u:(Z,c) \times (X,a) \rTo (S,\pi_2)$. By the corollary we have that $(z,g_1,g_2,y)$ belongs to $[(S,\pi_2)^{m \times Id_Y} \bar{\delta}]_{(Z,c)}(u)$ if and only if $(g_2^{-1}g_1^{-1}z,y)$ belongs to $\delta_Z(u)$. Now since $u$ is a $G$-homomorphism $u(z,gx)=u(g^{-1}z,x) $ and so by applying naturality of $\delta$ at $Z \times G \rTo^{c(i \pi_2, \pi_1)} Z$ we have $\delta_{Z \times G}( u ( Id_Z \times a))(z,g,y)=\delta_Z(u)(g^{-1}z,y)$. But then $\overline{\delta^G S^a}(u)$ is given by 
\begin{diagram}
Z \times G \times G \times Y  & \rTo & Z \times G \times Y & \rTo & Z \times Y  & \rTo^{\delta_Z(u)}  & S\\
(z,g_1,g_2,y) & \mapsto & (g_1^{-1}z,g_2,y) & \mapsto & (g_2^{-1}g_1^{-1}z,y) & & \\
\end{diagram}

\end{proof}

One of our categorical axioms, to follow, is that the category in question must be order enriched. Finite limits are assumed to be order enriched finite limits; that is, their universal property is an order isomorphism, not just a bijection. The above analysis works equally well with order isomorphisms in place of bijections; therefore,
\begin{proposition}\label{main}
Let $\mathcal{C}$ be an order enriched category with finite products, $G$ an internal group and $S$ a double exponentiable object. Then $(S,\pi_2)$ is a double exponentiable object in $[G,\mathcal{C}]$.
\end{proposition}

We need to discuss \emph{order internal} lattices in the context of an order enriched category; i.e. lattices such that the meet and join operations are adjoints to the diagonal (so being a lattice, join semilattice, meet semilattice, distributive lattice etc is a property of the object, not additional structure on the object). The following lemma will be needed:
\begin{lemma}\label{meet}
If $\mathcal{C}$ is an order enriched catrgory with finite products, then for any order internal meet semilattice $A$, if $A_0 \pile{ \rInto^i \\ \lOnto_q}A$ is a splitting of an inflationary idempotent $\psi: A \rTo A$ (i.e. $Id_A \sqsubseteq \psi=iq$ and $Id_{A_0}=qi$), then $A_0$ is an order internal meet semilattice and $i$ is a meet semilattice homomorphism. Further, $q$ preserves the top element (i.e. $q1_A = 1_{A_0}$).
\end{lemma}
\begin{proof}
Define $1_{A_0}  :  1 \rTo A_0$ to be $q1_A$ (so $q$ preserves top) and $\sqcap_{A_0} : A_0 \times A_0 \rTo A_0$ to be $q\sqcap_A (i \times i ) : A \times A \rTo A$. It can be verified that $!^{A_0} \dashv 1_{A_0}$ and $\Delta_{A_0} \dashv \sqcap_{A_0} $ and so $A_0$ is an order internal meet semilattice. To prove $i$ is a meet semilattice homomorphism we need to show (i) that $i$ preserves the top element and (ii) $iq\sqcap_{A}(i \times i) = \sqcap_A (i \times i)$. For (i) notice that $Id_A \sqsubseteq i 1_{A_0} !^A$ because $Id_A \sqsubseteq iq$ and so $i1_{A_0}=1_A$ by uniqueness of right adjoints. For (ii), as $Id_A \sqsubseteq iq$ it just needs to be checked that $iq\sqcap_{A}(i \times i) \sqsubseteq \sqcap_A (i \times i)$; equivalently, $\Delta_A iq \sqcap_A ( i \times i ) \sqsubseteq i \times i$ since $\Delta_A \dashv \sqcap_A$. But this last inequality is clear because $\Delta_A iq = (i \times i )(q \times q) \Delta_A$, $\Delta_A \sqcap_A \sqsubseteq Id_{A \times A}$ and $(i \times i )(q \times q) (i \times i)= i \times i $. 
\end{proof}

If, further, $\mathcal{C}$ has finite coproducts and is distributive (i.e. the canonical map $X \times Y + X \times Z \rTo X \times (Y + Z) $ is an isomorphism for any three objects $X$, $Y$ and $Z$ and $ X\times 0 \cong 0$ for any $X$) then for any object $S$, $\mathcal{C}^{op}_S$ has products; $S^X \times S^Y$ is given by $S^{X+Y}$ and the final object is $S^0$.

If additionally, $S$ is an order internal lattice and is double exponentiable then provided $\mathcal{C}$ has equalizers (and so is cartesian) two submonads of $P$ can be defined; a lower one, whose points are those natural transformations $S^X \rTo S^Y$ that are join semilattice homomorphisms and an upper one, whose points are meet semilattice homomorphisms. By reversing the order enrichment you switch between the lower and upper submonads. By construction the opposite of the Kleisli categories of the lower and upper monads can be identified with subcategories $\mathcal{C}_S^{op}$; they have the same objects and have as morphisms those natural transformations that are join (respectively meet) semilattice homomorphisms. Notice that all objects of the opposites of the Kleisli categories are order internal lattices which are distributive if $S$ is. See \cite{towhofman} for more detail on the construction of the lower and upper submonads.

Our final categorical definition is that of an object which behaves like the Sierpi\'{n}ski space. Given a cartesian order enriched category, an object $\mathbb{S}$ is a \emph{Sierpi\'{n}ski object} if it is an order internal distributive lattice such that given a pullback
\begin{diagram}
a^{\ast }(i) & \rTo & 1 \\
\dInto &  & \dInto_i \\
X & \rTo^{a} & \mathbb{S}
\end{diagram}
$a$ is uniquely determined by $a^{\ast }(i)\rTo X$ for $ i:1\rInto \mathbb{S}$ equal to either $0_{\mathbb{S}}$ or $1_{\mathbb{S}}$.
If a Sierpi\'{n}ski object is double exponentiable then we use $\mathbb{P}$ for the double exponential monad and call it a \emph{double power} monad; $P_L$ and $P_U$ are used for the lower and upper power monads,  when these can be defined as submonads of $\mathbb{P}$.

\section{The axioms}\label{axioms}
{\bf Axiom 1.}
\emph{$\mathcal{C}$ is an order enriched category with order enriched finite limits and finite coproducts.}

{\bf Axiom 2.}
\emph{For any morphism  $f:X \rTo Y$ of $\mathcal{C}$ the pullback functor $f^*  : \mathcal{C}/Y \rTo \mathcal{C}/X$ preserves finite coproducts.}

The property of being order enriched and having finite limits is $G$-stable, for any internal group $G$, as finite limits are created in $\mathcal{C}$ and the order enrichment on $[G,\mathcal{C}]$ can be taken from $\mathcal{C}$. Given Axiom 2, $[G,\mathcal{C}]$ has coproducts since if $(X,a)$ and $(Y,b)$ are two $G$-objects then 
\begin{eqnarray*}
G \times ( X+ Y) \rTo^{\cong} (G \times X ) + (G \times  Y )\rTo{a +b }  X+Y
\end{eqnarray*}
makes $X+Y$ into a $G$-object that can be easily checked to be coproduct. The nullary case is similar. If $f$ is a morphism of $G$-objects (i.e. a $G$-homomorphism) then pullback along $f$ preserves coproduct in $[G,\mathcal{C}]$ since $G$-object pullback and coproduct are created in $\mathcal{C}$. Therefore, 
\begin{lemma}
Axioms  1 and 2 are jointly $G$-stable for any internal group $G$.
\end{lemma}

{\bf Axiom 3.}
\emph{$\mathcal{C}$ has a Sierpi\'{n}ski object, $\mathbb{S}$.}

It is immediate that this axiom is $G$-stable, for any order enriched cartesian category $\mathcal{C}$, because pullbacks are created in $\mathcal{C}$. The canonical Sierpi\'{n}ski object in $[G,\mathcal{C}]$ is $(\mathbb{S},\pi_2)$.

{\bf Axiom 4.}
\emph{$\mathbb{S}$ is double exponentiable.}

That this axiom is $G$-stable follows from Proposition \ref{main}. Notice from the proposition that the morphisms of the Kleisli category $ [G , \mathcal{C} ]_{\mathbb{P}_G}$ can be identified with natural transformations $\delta: \mathbb{S}^X \rTo \mathbb{S}^Y$ with the property $\mathbb{S}^b\delta=\delta^G \mathbb{S}^a$. It is easy to see that the lower (upper) Kleisli maps correspond to $\delta $s that are join (meet) semilattice homomorphisms.

{\bf Axiom 5.}
\emph{For any objects $X$ and $Y$, any natural transformation $\alpha : \mathbb{S}^X\rTo \mathbb{S}^Y$ that is also a distributive lattice homomorphism, is of the form $\mathbb{S}^f$ for some unique $f:Y \rTo X$}

If $\epsilon:(S,\pi_2)^{(Y,b)} \rTo (S,\pi_2)^{(X,a)} $ is a natural transformtion that is also a distributive lattice homomorphism then the corresponding natural transformation $\delta: \mathbb{S}^X \rTo \mathbb{S}^Y$ is also a distributive lattice homomorphism and so, assuming the axiom, is equal to $\mathbb{S}^f$ for some unique $f: Y \rTo X$. However, by applying the uniqueness part of the axiom, we see that $f$ is a $G$-homomorphism. It is routine to then check, using the order isomorphism established in proposition \ref{main} that if $\delta$ is of the form $\mathbb{S}^f$ then $\epsilon$ must be $(\mathbb{S},\pi_2)^f$. This shows that the axiom is $G$-stable.

{\bf Axiom 6.}
\emph{(i) Inflationary idempotents split in the Kleisli category $\mathcal{C}_{P_L}$.}

\emph{(ii) Deflationary idempotents split in the Kleisli category $\mathcal{C}_{P_U}$.}

\cite{towhofman} shows that these conditions are equivalent to the assumption that the monad $P_L$ (respecitvely $P_U$) is KZ (respectively coKZ).

Say $\alpha: \mathbb{S}^X \rTo \mathbb{S}^X$ is an inflationary idempotent join semilattice homomorphism that splits as $ \mathbb{S}^{X_0}\pile{ \rInto^{\theta} \\ \lOnto_{\gamma}} \mathbb{S}^X$ in the (opposite of) the lower Kleisli category; so $\theta$ and $\gamma$ are both join semilattice homomorphisms. Then, in the presence of Axiom 5, $\theta$ must be equal to $\mathbb{S}^q$ for some unique $q$. This follows as lemma \ref{meet} shows that $\theta$ is a meet semilattice homomorphism. Notice also, by the `Further' part of that lemma, that $\gamma$ preserves top.

\begin{lemma}
Axiom 6 is $G$-stable (given Axioms 1-5).
\end{lemma}
In summary the proof follows by applying our description of the double power Kleisli morphisms of $[G,\mathcal{C}]$ in terms of the double power Kleisli morphisms of $\mathcal{C}$. 
\begin{proof}
If $(X,a)$ is a $G$-object and $\delta: \mathbb{S}^X \rTo \mathbb{S}^X$ an idempotent inflationary join semilattice homomorphism such that $\mathbb{S}^a\delta=\delta^G \mathbb{S}^a$, then $\delta$ factors as $\mathbb{S}^X \rOnto^{\gamma} \mathbb{S}^{X_0} \rInto^{\mathbb{S}^q} \mathbb{S}^X$; see the preamble to the statment of the lemma. Further $\delta^G$ factors as $\mathbb{S}^{Id_G \times q} \gamma^G$. Consider $\nu: \mathbb{S}^{X_0}\rTo \mathbb{S}^{G \times X_0}$ defined as $\gamma^G \mathbb{S}^a\mathbb{S}^q$. The two squares in the following diagram commute since $\gamma$ is a (split) epimorphism and $\mathbb{S}^{Id_G \times q}$ is a (split) monomorphism:
\begin{diagram}
\mathbb{S}^X & \rTo^{\gamma} & \mathbb{S}^{X_0} & \rTo^{\mathbb{S}^q } & \mathbb{S}^X \\
\dTo^{\mathbb{S}^{a}}  &  & \dTo_{\nu} &  & \dTo_{\mathbb{S}^a} \\
\mathbb{S}^{G \times X} & \rTo^{\gamma^G} & \mathbb{S}^{G \times X_0} & \rTo^{\mathbb{S}^{Id_G \times q}} & \mathbb{S}^{G \times X} \\
\end{diagram}
We now claim that $\nu$ is a meet semilattice homomorphism. To see this, by the `Further' part of Lemma \ref{meet}, we see that $\nu$ preserves top because both $\gamma$ and $\gamma^G$ preserve top. To establish preservation by $\nu$ of binary meets one needs but to check that $\sqcap_{\mathbb{S}^{G \times X_0}} (\nu \times \nu)\sqsubseteq \nu \sqcap_{\mathbb{S}^{X_0}}$. Now from Lemma \ref{meet} we know that $\sqcap_{\mathbb{S}^{X_0}} = \gamma \sqcap_{\mathbb{S}^X} ( \mathbb{S}^q \times \mathbb{S}^q)$ (and similarly for $\sqcap_{\mathbb{S}^{G \times X_0}}$). Therefore:
\begin{eqnarray*}
\nu \sqcap_{\mathbb{S}^{X_0}} & = & \gamma^G \mathbb{S}^a\mathbb{S}^q \gamma \sqcap_{\mathbb{S}^X} ( \mathbb{S}^q \times \mathbb{S}^q) \\
                                  & \sqsupseteq & \gamma^G \mathbb{S}^a  \sqcap_{\mathbb{S}^X} ( \mathbb{S}^q \times \mathbb{S}^q) \\  
                                   & = & \gamma^G   \sqcap_{\mathbb{S}^{G \times X}} (\mathbb{S}^a \times \mathbb{S}^a )( \mathbb{S}^q \times \mathbb{S}^q) \\  
                                  & = & \gamma^G  \sqcap_{\mathbb{S}^{G \times X} } (\mathbb{S}^{Id_G \times q}  \times \mathbb{S}^{Id_G \times q}) ( \nu \times \nu) \\        
                                  & = & \sqcap_{\mathbb{S}^{G \times X_0} }  ( \nu \times \nu) \text{.}\\       
\end{eqnarray*}
Since then $\nu$ is a distributive lattice homomorphism it is of the form $\mathbb{S}^t$ for some (unique) $t: G \times X_0 \rTo X_0$ and it is readily checked that $(X_0,t)$ is a $G$-object. By construction $\gamma$ and $\mathbb{S}^q$ commute with $\mathbb{S}^a$ and $\mathbb{S}^t$ and so correspond to morphisms of $[G,\mathcal{C}]^{op}_{\mathbb{P}_G}$ (i.e. natural transformations relative to the category of $G$-objects).
This proves stability of 6(i); part (ii) is order dual.
\end{proof}

{\bf Axiom 7.}
\emph{For any equalizer diagram }
\begin{diagram}
E & \rTo^{e} & X & \pile { \rTo^f \\ \rTo_g} & Y
\end{diagram}
in $\mathcal{C}$ the diagram
\begin{diagram}
\mathbb{S}^{X}\times \mathbb{S}^{X}\times \mathbb{S}^{Y} & \pile{ \rTo^{\sqcap (Id\times \sqcup )(Id\times Id\times \mathbb{S}^{f})} \\ \rTo_{\sqcap (Id\times \sqcup )(Id\times Id\times \mathbb{S}^{g})} } &  \mathbb{S}^{X} & \rTo{\mathbb{S}^{e}} & \mathbb{S}^{E}
\end{diagram}
\emph{is a coequalizer in $\mathcal{C}_{\mathbb{P}}^{op}$}.

Note that Axiom 7 does not break the symmetry given by the order
enrichment. A short calculation using the distributivity assumption on $%
\mathbb{S}$ shows that the composite $\sqcup (Id\times \sqcap )$ could have
been used in the place of $\sqcap (Id\times \sqcup )$.

Stability of this axiom is also straightforward as $\mathbb{S}^e$ is an epimorphism in $\mathcal{C}_{\mathbb{P}}^{op}$. In more detail say $(E,d)\rTo^e (X,a)$ is an equalizer of $f,g:(X,a) \pile{ \rTo \\ \rTo } (Y,b) $ in $[ G , \mathcal{C}]$ and $(Z,c)$ is a $G$-object, then for any $\delta : \mathbb{S}^X \rTo \mathbb{S}^Z$ which has $\mathbb{S}^c \delta = \delta^G \mathbb{S}^a$ we also have $\mathbb{S}^c \delta' \mathbb{S}^e= (\delta')^G \mathbb{S}^d\mathbb{S}^e$ if $\delta$ factors as $\delta'\mathbb{S}^e$ because $e$ is a $G$-homomorphism. $\delta'$ must then correspond to a morphism of $[G,\mathcal{C}]^{op}_{\mathbb{P}_G}$.

\begin{definition}
A category $\mathcal{C}$ satisfying the axioms is called a \emph{category of spaces}.
\end{definition}
\begin{example}
The category of locales relative to an elementary topos $\mathcal{E}$, written $\mathbf{Loc}_{\mathcal{E}}$, is a category of spaces. The axioms are all known properties of the category of locales; e.g. \cite{towaxioms} and \cite{towhofman}.
\end{example}

For clarity, collecting together the various observations already made:
\begin{theorem}
The axioms are $G$-stable for any internal group $G$; in other words if $\mathcal{C}$ is a category of spaces then so is $[G,\mathcal{C}]$ for any internal group $G$.
\end{theorem}

\section{Categories of spaces that are not categories of locales}
In this section we provide a class of examples which shows that not every category of spaces is a category of locales for some elementary topos $\mathcal{E}$. To give this example we must first recall a few basic definitions and results about categories of spaces and a proposition about the representation of geometric morphisms as certain adjunctions between categories of locales. 
\begin{definition}
(1) A morphism $f:X \rTo Y $ of a category of spaces is \emph{open} if there exists $\exists_f : \mathbb{S}^X \rTo \mathbb{S}^Y$ left adjoint to $\mathbb{S}^f$ such that $\exists_f\sqcap_{\mathbb{S}^X}(Id_{\mathbb{S}^X} \times \mathbb{S}^f)=\sqcap_{\mathbb{S}^Y}(\exists_f \times Id_{\mathbb{S}^Y})$ (Frobenius condition).
(2) An object $X$ of a category of spaces is \emph{open} if $!:X \rTo 1$ is an open map.
(3) An object $X$ of a category of spaces is \emph{discrete} if it is open and $\Delta : X \rTo X \times X$ is open.

\end{definition}
In the case where the category of spaces is a category of locales, the usual meanings are recovered; \cite{towaxioms}. Any elementary topos $\mathcal{E}$ can be identified with the full subcategory of $\mathbf{Loc}_{\mathcal{E}}$ consisting of discrete objects. One easily checks all isomorphism are open maps (notice: $\exists_{\phi^{-1}}=\mathbb{S}^{\phi}$ for any isomorphism $\phi$), and the property of being an open map is stable under composition, relative to any category of spaces; $\exists_{fg}=\exists_f\exists_g$ for any composable pair of morphisms $f$ and $g$. Further, open maps are pullback stable (\cite{towaxioms}) and the usual Beck-Chevalley condition holds for any pullback square (where an open map is being pulled back).

\begin{lemma}
If $\mathcal{C}$ is a category of spaces and $G=(G,m,e,i)$ is an internal group then a $G$-homomorphism $f:(X,a) \rTo (Y,b)$ is open relative to $[ G , \mathcal{C} ]$ if and only if $f: X \rTo Y $ is open relative to $\mathcal{C}$.
\end{lemma}
\begin{proof}
If $f$ is open as a $G$-homomorphism then there is a natural transformation $ (\mathbb{S}, \pi_2)^{(X,a)} \rTo (\mathbb{S}, \pi_2)^{(Y,b)} $ left adjoint to $(\mathbb{S}, \pi_2)^f$ and satisfying the Frobenius condition. But this natural transformation corresponds to a natural transformation $\mathbb{S}^X \rTo \mathbb{S}^Y $ which can be seen to witness that $f$ is open relative to $\mathcal{C}$. In the other direction if $f$ is open relative to $\mathcal{C}$ then there is $\exists_f : \mathbb{S}^X \rTo \mathbb{S}^Y $ left adjoint to $\mathbb{S}^f$ witnessing that $f$ is an open map of $\mathcal{C}$. So to complete the proof we can just check that the diagram
\begin{diagram}
\mathbb{S}^X & \rTo{\exists_f} & \mathbb{S}^Y \\
\dTo^{\mathbb{S}^a} &    &  \dTo_{\mathbb{S}^b} \\
\mathbb{S}^{G \times X} & \rTo{\exists_f^G} & \mathbb{S}^{G \times Y}\\ 
\end{diagram}
commutes, since then $\exists_f$ corresponds to a natural transformation  $ (\mathbb{S}, \pi_2)^{(X,a)} \rTo (\mathbb{S}, \pi_2)^{(Y,b)} $ relative to $[ G , \mathcal{C}]$, which can be seen to witness that $f$ is open as a $G$-homomorphism.

To prove that the square commutes, notice that $b: G \times Y \rTo Y $ factors as $G\times Y \rTo^{(\pi_1,b)} G \times Y \rTo^{\pi_2^Y} Y$ where the first factor is an isomorphism, and so
 
\begin{eqnarray*}
\mathbb{S}^b \exists_f & = & \mathbb{S}^{(\pi_1,b)} \mathbb{S}^{\pi_2^Y}\exists_f  \\
                        & = & \mathbb{S}^{(\pi_1,b)} \exists_{Id_G \times f}   \mathbb{S}^{\pi_2^X} \\
                        & = &  \exists_{(\pi_1,b(i \times Id_Y))} \exists_{Id_G \times f}   \mathbb{S}^{\pi_2^X} \\
                        & = &  \exists_{Id_G \times f} \exists_{(\pi_1,a(i \times Id_X))}   \mathbb{S}^{\pi_2^X} \\
                        & = &  \exists^G_ f   \mathbb{S}^{(\pi_1,a )}   \mathbb{S}^{\pi_2^X} \\
                        & = &  \exists^G_ f   \mathbb{S}^a \\
\end{eqnarray*}
where the second line is by Beck-Chevalley applied to the pullback square that is formed by pulling $f:X \rTo Y$ back along $\pi_2^Y : G \times Y \rTo Y$,  the third and fifth lines use $\exists_{\phi^{-1}}=\mathbb{S}^{\phi}$  for any isomorphism $\phi$, and the fourth line follows because $f$ is a $G$-homomorphism.

\end{proof}
If $G$ is a group in a category of spaces $\mathcal{C}$ then we use $BG$ for the full subcategory of $[ G , \mathcal{C}]$ consisting of discrete objects; the lemma can be applied to show that $BG$ is the full subcategory that consists of those $G$-objects $(X,a)$ such that $X$ is discrete relative to $\mathcal{C}$. So, in the case $\mathcal{C}=\mathbf{Loc}$, $BG$ recovers its usual meaning: $G$-sets. 

\begin{proposition}\label{geommorph}
Let $\mathcal{F}$ and $\mathcal{E}$ be two elementary toposes. There is an equivalence between the category of order enriched Frobenius adjunctions $L \dashv R : \mathbf{Loc}_{\mathcal{F}} \pile{\rTo \\ \lTo } \mathbf{Loc}_{\mathcal{E}}$ such that $R$ preserves the Sierp\'{n}ski locale and the category of geometric morphisms from $\mathcal{F}$ to $\mathcal{E}$. Every such Frobenius adjunction is determined up to isomorphism by the restriction of its right adjoint to discrete objects.
\end{proposition}
\begin{proof}
This is the main result of \cite{towgeom}.
\end{proof}
If $\mathcal{F} \rTo \mathcal{E}$ is a geometric morphism then we use $\Sigma_f \dashv f^*$ for the corresponding adjunction between categories of locales. We are now in a position to give our example.

\begin{example}
It is not the case that every category of spaces arises as the category of locales for some elementary topos. Let $G$ be a localic group, and say  $\psi: [ G , \mathbf{Loc}]  \rTo^{\simeq} \mathbf{Loc}_{\mathcal{E}}$ for some elelmentary topos $\mathcal{E}$ (such that the equivalence sends the Sierpi\'{n}ski locale relative to $\mathcal{E}$ to the canonical Sierpi\'{n}ski object of $[ G , \mathbf{Loc}]$). It follows that the discrete objects of $ \mathbf{Loc}_{\mathcal{E}}$ can be identified with the discrete objects of $[ G , \mathbf{Loc}]$; but these last are $BG$. It follows that $BG \simeq \mathcal{E}$ and therefore that there is an equivalence $\phi: \mathbf{Loc}_{\mathcal{E}} \rTo^{\simeq} \mathbf{Loc}_{BG}$. So there is an adjunction 
\begin{eqnarray*}
\mathbf{Loc} \pile{\rTo^{G \times (\_)}  \\ \lTo_U} [ G , \mathbf{Loc}] \pile{ \rTo^{\psi} \\ \lTo_{\psi^{-1}}} \mathbf{Loc}_{\mathcal{E}} \pile{ \rTo^{\phi} \\ \lTo_{\phi^{-1}}} \mathbf{Loc}_{BG}
\end{eqnarray*}
which satisfies Frobenius reciprocity and whose right adjoint preverse the Sierpi\'{n}ski locale. Further the restriction of the right adjoint of this adjunction to discrete locales is the forgetful functor and so by the last proposition this adjunction must be isomorphic to the adjunction $\Sigma_{p_G} \dashv p_G^*$ determined by the canonical point  $p_G : \mathbf{Set} \rTo BG$ of $BG$.

For any open localic group we know that the geometric morphism $p_G : \mathbf{Set} \rTo BG$ is an open surjection (see Lemma C5.3.6 of \cite{Elephant} and the comments before it). But locales decend along open surjections (Theorem C5.1.5 of \cite{Elephant}) and the definition of locales descending along $p_G$ is that the functor $\rho: \mathbf{Loc}_{BG} \rTo $$ [ \hat{G} , \mathbf{Loc}]$, induced by $p_G^* : \mathbf{Loc}_{BG} \rTo \mathbf{Loc}$ (i.e. $U\rho=p_G^*$), is an equivalence, where $\hat{G}$ is the \'{e}tale completion of $G$ (see e.g. Lemma C5.3.16 of \cite{Elephant} for a bit more detail). Therefore there exists an equivalence of categories $[ G , \mathbf{Loc} ] \simeq [ \hat{G} , \mathbf{Loc}]$ which commutes with the canonical adjunction back to $\mathbf{Loc}$. This is sufficient to show that $G \cong \hat{G}$ (see the comments before Lemma \ref{Frob}); i.e. that $G$ is \'{e}tale complete. Since not every open localic group is \'{e}tale complete, it is not the case that every category of spaces is a category of locales over some topos.
\end{example}

\section{Making do without coequalizers}
\subsection{Making do: inside $\mathcal{C}$}
An achievement of the axiomatic approach to locale theory is that it covers Plewe's result that localic triquotient surjections are effective descent morphisms (which generalises the more well known results that localic proper and open surjections are effective descent morphisms). To prove the result one needs to show that triquotient surjections are regular epimorphisms and, on the surface, this appears to require some coequalizers of the ambient category $\mathcal{C}$. We now show how to avoid this requirement.

\begin{definition}
Given a morphism $p: Z \rTo Y$ in a category of spaces, a \emph{triquotient assignment on $p$} is a natural trasnformation $p_{\#}: \mathbb{S}^Z \rTo \mathbb{S}^Y$ satisfying 

(i) $\sqcap_{\mathbb{S}^Y}( p_{\#} \times  Id_{\mathbb{S}^Y}) \sqsubseteq  p_{\#}\sqcap_{\mathbb{S}^Z} (Id_{\mathbb{S}^Z} \times \mathbb{S}^p  )$ and 

(ii)  $  p_{\#}\sqcup_{\mathbb{S}^Z} (Id_{\mathbb{S}^Z} \times \mathbb{S}^p  )\sqsubseteq \sqcup_{\mathbb{S}^Y}( p_{\#} \times  Id_{\mathbb{S}^Y})$.

Further $p$ is a \emph{triquotient surjection} if it has a triquotient assigment $p_{\#}$ such that $p_{\#}\mathbb{S}^p = Id_{\mathbb{S}^Y}$.
\end{definition}

Consult \cite{towaxioms} for more detail on triquotient assignments and the role they play in the axiomatic aporoach. In particular note that the usual `Beck-Chevalley for pullback squares' result holds: if $p_{\#}$ is a triquotient assignment on $p:Z \rTo Y$ then for any $f: X \rTo Y$ there is a triquotient assignment $(\pi_1)_{_\#}$ on $\pi_1  :X \times_Y Z \rTo X$ such that $(\pi_1)_{_\#}\mathbb{S}^{\pi_2} = \mathbb{S}^f p_{\#}$. Notice that if $p:Z \rTo Y$ is a triquotient surjection witnessed by the triquotient assignment $p_{\#}  : \mathbb{S}^Z \rTo \mathbb{S}^Y$, then $p_{\#}(1)=1$ and $p_{\#}(0)=0$. Conversely if $p: Z \rTo Y$ has a triquotient assignment $p_{\#}$ with $p_{\#}(1)=1$ and $p_{\#}(0)=0$ then $p_{\#}(\mathbb{S}^p(b))=p_{\#}(0 \sqcup \mathbb{S}^p(b))\sqsubseteq p_{\#}(0) \sqcup b = b$ and order dually $b \sqsubseteq p_{\#}(\mathbb{S}^p(b))$ and so $p$ is a triquotient surjection. Using this charcterization of triquotient surjection it is clear from Beck-Chevalley for pullback squares that triquotient surjections are pullback stable. We now prove that triquotient surjections are regular epimorphisms.
\begin{proposition}
If $\mathcal{C}$ is a category of spaces and $p: Z \rTo Y$ a triquotient surjection then $p$ is a regular epimorphism.
\end{proposition}
\begin{proof}
Let $p_1,p_2: Z \times_Y Z \pile{\rTo \\ \rTo} Z$ be the kernal pair of $p$. The diagram 
\begin{diagram}
\mathbb{S}^Y &  \pile{\rTo^{\mathbb{S}^p} \\ \lTo_{p_{\#}}} & \mathbb{S}^Z & \pile{ \rTo^{\mathbb{S}^{p_2}} \\  \rTo^{\mathbb{S}^{p_1} } \\ \lTo_{(p_1)_{\#}}} & \mathbb{S}^{Z \times_Y Z }\\
\end{diagram}
is a split fork in $\mathcal{C}_{\mathbb{P}}^{op}$. For any $q: Z \rTo W$ with $qp_1=qp_2$ we therefore have that $\mathbb{S}^q$ factors (uniquely) as $\mathbb{S}^p \alpha $ for some natural transformation $\alpha$ (it is given by $p_{\#}\mathbb{S}^q$). By Axiom 5 it therefore only remains to check that $\alpha$ is a distributive lattice homomorphism. Since we have already observed $p_{\#}$ preserves $0$ and $1$ we just need to show that $\alpha$ preserves binary meet and join, and for this it is sufficient to check $p_{\#}\mathbb{S}^q(c_1) \sqcap p_{\#}\mathbb{S}^q(c_2) \sqsubseteq p_{\#}\mathbb{S}^q(c_1 \sqcap c_2)$ and $  p_{\#}\mathbb{S}^q(c_1 \sqcup c_2) \sqsubseteq p_{\#}\mathbb{S}^q(c_1) \sqcup p_{\#}\mathbb{S}^q(c_2)$. 
But 
\begin{eqnarray*}
p_{\#}\mathbb{S}^q(c_1) \sqcap p_{\#}\mathbb{S}^q(c_2) & \sqsubseteq & p_{\#}(\mathbb{S}^p c_1 \sqcap  \mathbb{S}^p p_{\#} \mathbb{S}^q c_2) \\
& = & p_{\#}(\mathbb{S}^q c_1 \sqcap (p_1)_{\#} \mathbb{S}^{p_2} \mathbb{S}^q c_2) \text{ (Beck-Chevalley)}\\
& = & p_{\#}(\mathbb{S}^q c_1 \sqcap (p_1)_{\#} \mathbb{S}^{p_1} \mathbb{S}^q c_2) \text{ (since $qp_1=qp_2$)} \\
& = & p_{\#}(\mathbb{S}^q c_1 \sqcap \mathbb{S}^q c_2) \text{ ($p_1$ triquotient surjection)}\\
& = & p_{\#}\mathbb{S}^q ( c_1 \sqcap c_2)\\
\end{eqnarray*}
and $  p_{\#}\mathbb{S}^q(c_1 \sqcup c_2) \sqsubseteq p_{\#}\mathbb{S}^q(c_1) \sqcup p_{\#}\mathbb{S}^q(c_2)$ follows an order dual proof and so we are done.
\end{proof}
Further details on the axiomatic proof that triquotient surjections are of effective decent are contained in \cite{towaxiomsdraft}.

\subsection{Making do: maps between $\mathcal{C}$s}
If $G$ is an internal group in a category of spaces we have established that $[G,\mathcal{C}]$ is a category of spaces. Since we have also recalled in Proposition \ref{geommorph} that geometric morphisms can be represented as certain adjunctions between categories of spaces it would be odd if there was not a natural `connected components' adjunction $\Sigma_G \dashv G^*$; i.e.
\begin{eqnarray*}
[G,\mathcal{C}] \pile{\rTo^{\Sigma_G} \\ \lTo_{G^*} } \mathcal{C}
\end{eqnarray*}
where $G^*$ sends an object $X$ of $\mathcal{C}$ to the trivial $G$-object $(X,\pi_2)$. But for $\Sigma_G$ to exist it would appear that coequalizers are required, since $\Sigma_G(X,a)$ must (by uniqueness of left adjoints) be isomorphic to the coequalizer of $a$ and $\pi_2$. We now show, for open groups at least, that in fact $\Sigma_G$ can always be defined. (Order dually, $\Sigma_G$ will always exist for compact groups.) The proof does not require $G$ to be a group, only a monoid, but it is not clear what this extra level of generality offers us. To prove this result we need three lemmas, the first of which is a simple order enriched result:
\begin{lemma}\label{hannahs}
If
\begin{eqnarray*}
C \rTo^c A \pile{\rTo^a \\ \rTo_b } B
\end{eqnarray*}
is a fork diagram in an order enriched category $\mathcal{C}$ (i.e. $ac=bc$), then if there exists $q:A\rTo C$ and $t:B\rTo A$ such that $ta=cq \sqsupseteq Id_A$, $qc=Id_C$ and $tb \sqsubseteq Id_A$, then $c$ is the equalizer of $a$ and $b$. 
\end{lemma}
The result that  $c$  is an equalizer is similar to the familiar result that split forks are coequalizers, used in Beck's monadicity theorem. An order enriched monadicity theorem can be written down, based on this result.
\begin{proof}
Say $d:D \rTo A$ has $ad=bd$. Then $cqd \sqsupseteq d$ and $cqd = tad = tbd \sqsubseteq d$; so, $cqd=d$ showing that $d$ factors via $c$, clearly uniquely as $c$ is split. This establishes an order isomorphism as the action of morphism composition preserves order. 
\end{proof}

For the next lemma observe that if $p:Z \rTo X$ is an open map with a section $s: X \rTo Z$ (i.e. $ps=Id_X$) then $\mathbb{S}^p \sqsubseteq \exists_p$. This is trivial to establish because $Id_{\mathbb{S}^Z} \sqsubseteq \mathbb{S}^p \exists_p$ since $\mathbb{S}^p$ is right adjoint to $\exists_p$. In particular we observe that for any open object $Y$, arbitrary object $X$, and map $g:X \rTo Y$, we have that $\mathbb{S}^{(g,Id_X)} \sqsubseteq \exists_{\pi_2}$, where $\pi_2 : Y \times X \rTo X$, which is open because it is the pullback of the open map $!:Y \rTo 1$. Our next lemma builds on this last observation.
\begin{lemma}
If $g: Z_1 \rTo Z_2$ is a map between two open objects and $X$ is some other object of $\mathcal{C}$, then $\exists_{\pi_2^{Z_1}}\mathbb{S}^{g \times Id_X} \sqsubseteq \exists_{\pi_2^{Z_2}}$ (where $\pi_2^{Z_1}: Z_1 \times X \rTo X$ and $\pi_2^{Z_2}: Z_2 \times X \rTo X$).
\end{lemma}
\begin{proof}
As $\exists_{\pi_2^{Z_1}}$ is left adjoint to $\mathbb{S}^{\pi_2^{Z_1}}$, the proof can be completed by showing $\mathbb{S}^{g \times Id_X} \sqsubseteq \mathbb{S}^{\pi_2^{Z_1}}\exists_{\pi_2^{Z_2}}$, which is equivalent to showing $\mathbb{S}^{g \times Id_X} \sqsubseteq \exists_{\pi_{23}} \mathbb{S}^{\pi_{13}}$ by Beck-Chevalley on the pullback square
\begin{diagram}
Z_2 \times Z_1 \times X & \rTo^{\pi_{13}} & Z_2 \times X \\
\dTo^{\pi_{23}} &  &  \dTo_{\pi_2^{Z_2}} \\
Z_1 \times X  & \rTo_{\pi_2^{Z_1}} & X \\
\end{diagram}
But the proof is then complete by our observations in the preamble because $\mathbb{S}^{g \times Id_X}$ factors as $\mathbb{S}^{(g,Id_{Z_1 \times X})}\mathbb{S}^{\pi_{23}}$.
\end{proof}
This leads us to our first result about open monoids; that is, on monoid objects $(M,m:M\times M \rTo M,e: 1 \rTo M)$, internal to $\mathcal{C}$, such that $M$ is open.
\begin{lemma}
If $M$ is an open monoid then for any $M$-object $(X,a:M \times X \rTo X)$, $\mathbb{S}^X\rTo^{\mathbb{S}^a}\mathbb{S}^{M \times X} \rTo^{\exists_{\pi_2}}\mathbb{S}^X$ is (a) inflationary and (b) idempotent.
\end{lemma}
\begin{proof}
(a) Immediate because $\exists_{\pi_2}$ is greater than $\mathbb{S}^{(e!,Id_X)}$ and $Id_{\mathbb{S}^X}$ factors as $\mathbb{S}^{(e!,Id_X)}\mathbb{S}^a$.

(b) By Beck-Chevalley on the pullback square

\begin{diagram}
M \times M \times X &  \rTo^{Id_M \times a} & M \times X \\
\dTo^{\pi_{23}}  &    &  \dTo_{\pi_2} \\
M \times X & \rTo^a   & X  \\
\end{diagram}
and using $a(Id_M \times a) = a (m \times Id_X)$ we have 
\begin{eqnarray*}
\exists_{\pi_2} \mathbb{S}^a \exists_{\pi_2} \mathbb{S}^a & = & \exists_{\pi_2} \exists_{\pi_{23}} \mathbb{S}^{Id_M
 \times a}  \mathbb{S}^a  \\
& = & \exists_{\pi_2} \exists_{\pi_{23}} \mathbb{S}^{m \times Id_X}  \mathbb{S}^a  \\
& \sqsubseteq & \exists_{\pi_2} \mathbb{S}^a
\end{eqnarray*}
where the last line is by the lemma (take $g=m$). This completes the proof of (b), given that (a) shows that $\exists_{\pi_2} \mathbb{S}^a$ is inflationary.

\end{proof}

The next result establishes the aim of this subsection. 
\begin{proposition}\label{nathan}
If $M$ is an open monoid then $M^*: \mathcal{C} \rTo $$ [ M , \mathcal{C}]$ has a left adjoint, $\Sigma_M $. 
\end{proposition}
\begin{proof}
For any $M$-object $(X,a)$ consider the map $\exists_{\pi_2}\mathbb{S}^a$, which we have established is an inflationary idempotent. So by applying by Axioms 5 and 6 we have a diagram 
\begin{eqnarray*}
\mathbb{S}^{\Sigma_M(X,a)} \pile{ \rInto^{\mathbb{S}^{q^X}} \\ \lOnto_{\tau} } \mathbb{S}^X \pile{ \rTo^{\mathbb{S}^a} \\ \rTo^{\mathbb{S}^{\pi_2}} \\ \lTo_{\exists_a}} \mathbb{S}^{G \times X}
\end{eqnarray*}
which, by Lemma \ref{hannahs}, is an equalizer in $\mathcal{C}^{op}_{P_L}$. From this it follows that $q^X$ is the coequalizer of $a$ and $\pi_2$; if $t: X \rTo Z$ composes equally with $a$ and $\pi_2$ then $\mathbb{S}^t$ factors uniquely via $\mathbb{S}^{q^X}$ so it just remains to check, as in earlier proofs, that $\tau \mathbb{S}^t$ is a meet semilattice homomorphism. Certainly it preserves top (as $\tau$ does); the manipulation below shows that $\mathbb{S}^{q^X}( \tau \mathbb{S}^t (a_1) \sqcap \tau \mathbb{S}^t (a_2) ) \sqsubseteq \mathbb{S}^{q^X} \tau \mathbb{S}^t ( a_1 \sqcap a_2)$ from which it is clear that $\tau \mathbb{S}^t$ is a meet semilattice homomorphism (post compose the inequality with $\tau$).
\begin{eqnarray*}
\mathbb{S}^{q^X}( \tau \mathbb{S}^t (a_1) \sqcap \tau \mathbb{S}^t (a_2) ) & = & \mathbb{S}^{q^X} \tau \mathbb{S}^t (a_1) \sqcap  \mathbb{S}^{q^X} \tau \mathbb{S}^t (a_2) ) \\
& = & \exists_a \mathbb{S}^{\pi_2} \mathbb{S}^t (a_1) \sqcap  \exists_a \mathbb{S}^{\pi_2} \mathbb{S}^t (a_2) \\
& = & \exists_a \mathbb{S}^a \mathbb{S}^t (a_1) \sqcap  \exists_a \mathbb{S}^a \mathbb{S}^t (a_2) \\
& \sqsubseteq &  \mathbb{S}^t (a_1) \sqcap  \mathbb{S}^t (a_2) \\
& = &  \mathbb{S}^t (a_1 \sqcap  a_2) \\
& \sqsubseteq &  \mathbb{S}^{q^X}\tau \mathbb{S}^t (a_1 \sqcap a_2)\\
\end{eqnarray*}
\end{proof}
\section{Extending to groupoids}
In this section we outline how the above arguments extend to groupoids. We start by establishing some notation for slice categories. If $f : Y \rTo X$ is a morphism of a category $\mathcal{C}$ then we use $Y_f$ as notation for $f$ when considered as an object of the slice category $\mathcal{C}/X$. We use $Y_X$ as notation for the object $\pi_2 : Y \times X \rTo X $. Now any morphism $g:Z \rTo X$ of a cartesian category $\mathcal{C}$ gives rise to a pullback adjunction $\Sigma_g \dashv g^*: \mathcal{C}/Z \rTo \mathcal{C}/X$ that satisfies Frobenius reciprocity. So by the change of base result (Lemma \ref{Frob}) there is an adjunction, which we will write $g^{\#} \dashv g_*$, between $(\mathcal{C}/Z)_{S_Z}^{op} $ and $(\mathcal{C}/X)_{S_X}^{op} $. In the case that $X=1$ observe that for any $\delta : S^A \rTo S^B$ we have $g_*g^{\#}(\delta)=\delta^X$. 

If $\mathbb{G}=(G_1 \pile{\rTo^{d_0} \\ \rTo_{d_1} } G_0, m: G_1 \times_{G_0} G_1 \rTo G_1, s:G_0 \rTo G_1, i: G_1 \rTo G_1)$ is a groupoid relative to an order enriched cartesian category $\mathcal{C}$, with object of objects $G_0$ and object of morphisms $G_1$, then there is an adjunction $\mathcal{C}/G_0 \pile{ \rTo^{\Sigma_{d_1} d_0^*}\\ \lTo_U} [ \mathbb{G} ,\mathcal{C}]$. It satisfies Frobenius reciprocity and $U$ is monadic. A $\mathbb{G}$-object consists of $(X_f, a: \Sigma_{d_1} d_0^* X_f \rTo X_f)$ where $f: X \rTo G_0$ and $a$ is a morphism over $G_0$ that satisfies the usual unit and associative identities (the domain of $\Sigma_{d_1} d_0^* X_f$ is $G_1 \times_{G_0} X$). By taking adjoint transpose across $\Sigma_{d_1} \dashv d_1^*$ it is well known that having such an $a$ on $X_f$ is equivalent to having a morphism $a': d_0^* X_f \rTo d_1^* X_f$ of $\mathcal{C}/G_1$ such that $s^*a'$ is isomorphic to the identity and $m^*a' \cong \pi_2^*a' \nu_{X_f}\pi_1^*a'$ (where $\nu:\pi_1^*d_1^* \rTo^{\cong} \pi_2^* d_0^*$ is the canonical isomorphism). If $(X_f,a)$ and $(Y_g,b)$ are two $\mathbb{G}$-objects then a morphism $h:X_f\rTo Y_b$ of $\mathcal{C}/G_0$ is a $\mathbb{G}$-homomorphism if and only if $(d_1^*h)a'=b'(d_0^*h)$. 

For any object $S$ of $\mathcal{C}$, $S_{G_0}$ can be made into a $\mathbb{G}$-object by defining the trivial action on it: $Id_S \times d_1 : S \times G_1 \rTo S \times G_0$. This $\mathbb{G}$-object is written $\mathbb{G}^*S$.

To summarise the technical difference we have to account for when generalising from groups to groupoids, observe that the role of $\delta^G$, i.e. $!^G_* (!^G)^{\#}(\delta)$, must be taken by $({d_1})_*d_0^{\#}(\delta)$. So the exponential in the presheaf category, something that is determined by the categorical structure of $\mathcal{C}$ alone, must be replaced by an endofunctor relative to $\mathcal{C}$ that contains information about the groupoid. However, once this replacement is made it is easy to see how to make the generalisation. 

It does not appear to be possible to generalise Proposition \ref{main} from groups to groupoids. If it were possible then the property of being double exponentiable would be stable under slicing, since for any object $X$ the category of $\mathbb{G}$-objects is the same as the slice of $\mathcal{C}$ over $X$ when $\mathbb{G}$ is taken to be the trivial groupoid $X \pile{\rTo^{Id_X} \\ \rTo_{Id_X} } X$. But, see \cite{towslice}, proving slice stabity of double exponentiability appears to require something like Axiom 7 (or, at least, that the double exponentiation functor preserves coreflexive equalizers). 

\begin{proposition}
Let $S$ be a double exponentiable object in an order enriched cartesian category $\mathcal{C}$ such that double exponentiation is stable under slicing. For any internal groupoid $\mathbb{G}$, $\mathbb{G}^*S$ is a double exponentiable object of $[ \mathbb{G} ,\mathcal{C}]$. 
\end{proposition}
By `stable under slicing' we mean that for any morphism $g : Z \rTo X$ the canonical morphism $g^*\mathbb{P}_X \rTo \mathbb{P}_Z g^*$, determined by the fact that $\Sigma_g \dashv g^*$ satisfies Frobenius reciprocity, is an isomorphism 
\begin{proof}
Let $(X_f,a)$ be a $\mathbb{G}$-object. Then $P_{G_0}(X_f)$ can be made into a $\mathbb{G}$-object by using 
\begin{eqnarray*}
d_0^*P_{G_0}(X_f) \rTo^{\cong} P_{G_1}d_0^*(X_f) \rTo^{P_{G_1}a'} P_{G_1}d_1^*(X_f) \rTo^{\cong} d_1^*P_{G_0}(X_f)\text{.}
\end{eqnarray*}
$\mathbb{G}$-homomorphisms from $(Y_g, c)$ to $P_{G_0}(X_f)$ correspond to natural transformations $\delta : S_{G_0}^{X_f} \rTo S_{G_0}^{Y_g}$ such that $({d_1})_*d_0^{\#}(\delta)S_{G_0}^{a}=S_{G_0}^{b}\delta$ (equivalently $d_0^{\#}(\delta)S_{G_1}^{a'}=S_{G_1}^{b'}d_1^{\#}\delta$, by adjoint transpose via $d_1^{\#} \dashv (d_1)_*$). Since $\mathcal{C}/G_0 \pile{ \rTo^{\Sigma_{d_1} d_0^*}\\ \lTo_U} [ \mathbb{G} ,\mathcal{C}]$ satisfies Frobenius reciprocity we know that there is an order isomorphism 
\begin{eqnarray*}
Nat[S_{G_0}^{X_f},S_{G_0}^{Y_g}] \cong Nat[(\mathbb{G}^*S)^{(X_f,a)},(\mathbb{G}^*S)^{(G_1 \times_{G_0} Y,m \times Id_Y)}]
\end{eqnarray*}
natural in $Y_g$ and so the result follows as in Section \ref{prelim} from the explicit description of this order isomorphism and application of Lemma \ref{hannah} with $\mathbb{T}$ the monad induced by $ \Sigma_{d_1} d_0^* \dashv U$.
\end{proof}
Let us recall the slice stability result that we need to proceed. The proof is clear from \cite{towslice}.
\begin{proposition}
If $\mathcal{C}$ is a category of spaces then for any object $X$, so is $\mathcal{C}/X$. The canonical Sierpi\'{n}ski object relative to $\mathcal{C}/X$ is $\mathbb{S}_X$ and pullback commutes with double exponentiation.
\end{proposition}
By combining these last two propositions we have that $\mathbb{G}^*\mathbb{S}$ is a double exponentiable object in $[ \mathbb{G} ,\mathcal{C}]$, if $\mathcal{C}$ is a category of spacees. Checking that the remaining axioms are $\mathbb{G}$-stable is a straightforward re-application of the arguments deployed in Section \ref{axioms}, given that \cite{towslice} shows that the axioms hold in $\mathcal{C}/G_0$. So we have shown in outline:
\begin{theorem}
If $\mathcal{C}$ is a category of spaces and $\mathbb{G}$ an internal groupoid, then $[\mathbb{G},\mathcal{C}]$ is a category of spaces.
\end{theorem}



\begin{thebibliography}{test}

\bibitem[J02]{Elephant}  Johnstone, P.T. \emph{Sketches of an elephant: A
topos theory compendium}. Vols 1, 2, Oxford Logic Guides \textbf{43}, \textbf{44}, Oxford Science Publications, 2002.

\bibitem[JV91]{PrePrePre}  Johnstone, P.T., and Vickers, S.J. ``Preframe
presentations present'', in Carboni, Pedicchio and Rosolini (eds) Category
Theory -- Proceedings, Como, 1990 (\emph{Springer Lecture Notes in
Mathematics} \textbf{1488}, 1991), 193-212.

\bibitem[M88]{MoerClassTop} Moerdijk, I. \emph{The classifying topos of a continuous groupoid. I}, Transactions of the American Mathematical Society, Volume 310, Number \textbf{2}, (1988) 629-668.

\bibitem[JT84]{JoyT}  Joyal, A. and Tierney, M. \emph{An Extension of the Galois Theory of Grothendieck}, Memoirs of the American Mathematical Society
\textbf{309}, 1984.

\bibitem[P97]{triquot}  Plewe, T. \emph{Localic triquotient maps are
effective descent maps}, Math. Proc. Cambridge Philos. Soc. \textbf{122}
(1997), 17-43.

\bibitem[T12]{towslice} Townsend, C.F. \emph{Aspects of slice stability in Locale Theory} Georgian Mathematical Journal. Vol. 19, Issue \textbf{2}, (2012) 317–374.

\bibitem[T10a]{towaxioms}  Townsend, C.F. \emph{An Axiomatic account of Weak Localic Triquotient Assignments.} Journal of Pure and Applied Algebra Volume 214 \textbf{6} (2010) 729-739.

\bibitem[T10b]{towgeom} Townsend, C.F. \emph{A representation theorem for geometric morphisms.} Applied Categorical Structures. \textbf{18} (2010) 573-583.

\bibitem[T07]{closedsubgroup}  Townsend, C.F. \emph{A categorical account of the localic closed subgroup theorem. }Comment. Math. Univ. Carolin. \textbf{48}, 3 (2007) 541-553.

\bibitem[T05]{towhofman}  Townsend, C.F. \emph{A categorical account of the Hofmann-Mislove theorem. }Math. Proc. Camb. Phil. Soc. \textbf{139} (2005) 441-456.

\bibitem[T04]{towaxiomsdraft}  Townsend, C.F. \emph{An Axiomatic account of Weak Localic Triquotient Assignments. Early draft.} www.christophertownsend.org.

\bibitem[V95]{Vicpoints} Vickers, S.J. \emph{Locales are not pointless} in `Theory and Formal Methods 1994'. Second Imperial College Department of Computing Workshop on Theory and Formal Methods. Cambridge. 1994. Imperical College Press, London, (1995) 199-216.

\bibitem[V02]{DoubPt}  Vickers, S.J. \emph{The double powerlocale and
exponentiation: a case study in geometric logic}. Theory and Applications of
Categories \textbf{12} (2004) 372-442.

\bibitem[VT04]{victow}  Vickers, S.J. and Townsend, C.F. \emph{A Universal
Characterization of the Double Power Locale}. Theoretical Computer Science
\textbf{316 }(2004) 297-321.
\end{thebibliography}
\end{document}